\documentclass[10pt, reqno, english]{amsart}  
\usepackage[utf8]{inputenc}
\usepackage[T1]{fontenc}
\usepackage{amsmath,amsthm}
\usepackage{amsfonts,amssymb}
\usepackage{url}
\usepackage{mathtools}  
\usepackage[colorlinks=true,urlcolor=blue,linkcolor=red,citecolor=magenta]{hyperref}
\usepackage{enumerate, paralist}
\usepackage{tikz-cd}
\usepackage{xcolor}
\usepackage{hyperref}
\usepackage{caption}

\usepackage{soul}

\newtheoremstyle{myremark} 
{7pt}                    
{7pt}                    
{}  	                 
{}                           
{\bf}       	         
{.}                          
{.5em}                       
{}  

\makeatletter
\newtheorem*{rep@theorem}{\rep@title}
\newcommand{\newreptheorem}[2]{%
\newenvironment{rep#1}[1]{%
 \def\rep@title{#2 \ref{##1}}%
 \begin{rep@theorem}}%
 {\end{rep@theorem}}}
\makeatother

\theoremstyle{plain}
\newtheorem{lemma}{Lemma}[section]
\newtheorem{theorem}{Theorem}
\newreptheorem{theorem}{Theorem}
\newtheorem{corollary}[lemma]{Corollary}

\theoremstyle{definition}
\newtheorem{definition}[lemma]{Definition}
\newtheorem{question}[lemma]{Question}

\theoremstyle{myremark}
\newtheorem{remark}[lemma]{Remark}
\newtheorem{example}[lemma]{Example}

\newcommand{\R}{\mathbb{R}}
\newcommand{\Z}{\mathbb{Z}}
\newcommand{\F}{\mathbb{F}}

\newcommand{\diam}{\mathrm{diam}}

\newcommand{\conv}{\mathrm{conv}}

\newcommand{\sm}{\mathrm{SM}}

\newcommand{\SM}{\mathrm{SM}}

\newcommand{\vr}[2]{\mathrm{VR}(#1;#2)}

\newcommand{\gh}{\mathrm{GH}}

\newcommand{\supp}{\mathrm{supp}}

\newcommand{\ind}{\mathrm{ind}}
\newcommand{\coind}{\mathrm{coind}}

\makeatletter
\@namedef{subjclassname@2020}{%
  \textup{2020} Mathematics Subject Classification}
\makeatother

\begin{document}

\title[Projective codes and odd maps]{The topology of projective codes \\ and the distribution of zeros of odd maps}

\author{Henry Adams}
\address[HA]{Department of Mathematics, Colorado State University, Fort Collins, CO 80523, USA}
\email{henry.adams@colostate.edu}

\author{Johnathan Bush}
\address[JB]{Department of Mathematics, University of Florida, Gainesville, FL 32611, USA}
\email{bush.j@ufl.edu}

\author{Florian Frick}
\address[FF]{Dept.\ Math.\ Sciences, Carnegie Mellon University, Pittsburgh, PA 15213, USA\newline \indent Inst. Math., Freie Universit\"at Berlin, Arnimallee 2, 14195 Berlin, Germany}
\email{frick@cmu.edu}



\thanks{FF was supported by NSF grant DMS 1855591 and a Sloan Research Fellowship.}


\begin{abstract}
\small
We show that the size of codes in projective space controls structural results for zeros of odd maps from spheres to Euclidean space.
In fact, this relation is given through the topology of the space of probability measures on the sphere whose supports have diameter bounded by some specific parameter.
Our main result is a generalization of the Borsuk--Ulam theorem, and we derive four consequences of it:
(i)~We give a new proof of a result of Simonyi and Tardos on topological lower bounds for the circular chromatic number of a graph;
(ii)~we study generic embeddings of spheres into Euclidean space and show that projective codes give quantitative bounds for a measure of genericity of sphere embeddings;
and we prove generalizations of (iii)~the Ham Sandwich theorem and (iv)~the Lyusternik--Shnirel'man--Borsuk covering theorem for the case where the number of measures or sets in a covering, respectively, may exceed the ambient dimension.
\end{abstract}

\date{\today}
\maketitle

\section{Introduction}

Given a compact metric space $(X,d)$, a central problem in coding theory is to determine a configuration of $n$ points in $X$ that maximizes the minimal distance between them. 
More precisely, for each $n\geq 1$, we define \[p_n(X) \coloneqq \sup \{\varepsilon > 0~:~\text{there are} \ x_1, \dots, x_n \in X \ \text{with} \ d(x_i,x_j) \ge \varepsilon \ \text{for all} \ i \ne j\}.\]
By compactness, one can find  $\{x_1,\dots, x_n\}\subseteq X$ such that $d(x_i,x_j)\geq p_n(X)$ whenever $i\neq j$; such a set is called a \textit{code of size $n$}. 

For the $d$-sphere $S^d$, a code $\{x_1,\dots, x_n\}\subset S^d$ is called a \textit{spherical code}. 
Spherical codes have generated considerable interest; 
see for example~\cite{thompson1983error}.

Spheres are naturally equipped with the $\Z/2$-action exchanging antipodal points, denoted $x\mapsto -x$, and a subset of a sphere that is invariant under this action is said to be \emph{centrally-symmetric}.
Hence, one may also consider the problem of finding centrally-symmetric spherical codes, that is, spherical codes in which points come in antipodal pairs $\{x,-x\}$~\cite{bukh2020,cohn2016optimal}.
By equipping real projective space $\R P^d$ with the quotient metric induced by the antipodal action, a centrally-symmetric spherical code in $S^d$ is equivalent to a code in $\R P^d$, i.e.\ a \textit{projective code}. 
With this metric, the distance between two points in $\R P^d$ is the angle between the lines they determine in~$\R^{d+1}$.
In other words, a projective code is a collection of lines through the origin in~$\R^{d+1}$ that are as close to orthogonal as possible.
Therefore, the problem of finding a projective code of size $n$ may be rephrased as follows: Given $c>0$, find $n$ unit vectors $x_1,\dots, x_n\in S^d\subset \R^{d+1}$ such that $|\langle x_i,x_j\rangle |\leq c$ whenever $i\neq j$. 

In the present manuscript, we establish relationships between projective codes and the following two problem areas:
\begin{compactenum}[(1)]

\item \emph{Metric thickenings}: For a compact metric space $X$ and $\delta > 0$, determine bounds for the topology of the space of probability measures on $X$ supported on finite sets of diameter at most~$\delta$, equipped with an optimal transport metric.
This is the \emph{metric thickening $X_\delta$} at scale~$\delta$.

\item\label{it:zeros}
\emph{Zeros of odd maps}: A map $f\colon S^d \to \R$ is odd if $f(-x) = -f(x)$ for all~$x\in S^d$.
Let $V$ be an $n$-dimensional vector space of continuous odd maps $S^d \to \R$.
Find a small $\delta \ge 0$ such that there is a set $A \subset S^d$ of diameter at most~$\delta$ with the property that for every $f \in V$ the restriction $f|_A$ is not strictly positive.
\end{compactenum}

Metric thickenings were introduced in~\cite{AAF} to capture the geometry of a metric space at a given scale~${\delta > 0}$.
Here we show that large projective codes provide lower bounds for the topology of metric thickenings of spheres (in terms of their cohomological index), which in turn give non-trivial upper bounds for $\delta$ in problem area~(\ref{it:zeros}), and thus structural results for zeros of odd maps.
Our main result is:

\begin{theorem}
\label{thm:main}
Let $V$ be an $n$-dimensional vector space of continuous odd maps $S^d \to \R$.
Then there is a set $A \subset S^d$ of diameter at most $\delta = \pi-p_{n+1}(\R P^d)$ such that for every $f \in V$ the restriction $f|_A$ is not strictly positive.
\end{theorem}

Thus if $\R P^d$ has a close packing by metric balls, then odd maps $S^d \to \R$ cannot stay positive on large subsets of~$S^d$.
One can regard Theorem~\ref{thm:main} as
an asymptotic Borsuk--Ulam theorem.
To recover the classical Borsuk--Ulam theorem that any odd map $S^n\to\R^n$ has a zero, one must instead phrase Theorem~\ref{thm:main} in terms of the topology of metric thickenings of the sphere; see Remark~\ref{rmk:gen-BU}.
This theorem improves on a recent Borsuk--Ulam result for odd maps $S^d\to\R^n$ with $n\ge d$, proven by the authors in~\cite{ABF}.
We also refer the reader to~\cite{crabb2021borsuk}, which uses characteristic classes to give related generalizations of the Borsuk--Ulam theorem into higher-dimensional codomains.
The Borsuk--Ulam theorem has found numerous applications across mathematics; see for instance~\cite{matousek2003using} for applications in combinatorics and discrete geometry. 

An equivalent formulation of the Borsuk--Ulam theorem states that there does not exist a continuous odd map $S^n\to S^d$ when $n>d$. 
However, the obstruction to the existence of such maps can be eliminated by thickening the codomain.
The following theorem makes this precise and establishes a relationship between the topology of metric thickenings of spheres and projective codes. 

\begin{theorem}
\label{thm:lower-bound}
Let $d$ and $n$ be positive integers, and let $\delta \ge \pi-p_{n+1}(\R P^d)$.
Then there is a continuous odd map $S^n \to S^d_\delta$.
\end{theorem}

Theorem~\ref{thm:lower-bound} gives lower bounds for the topology of metric thickenings of spheres, $S^d_\delta$, in terms of packings of metric balls in projective space.
Dually, coverings of projective space by metric balls provide upper bounds for the topology of metric thickenings in Theorem~\ref{thm:covering}.
Thus the topology of the metric thickening $S^d_\delta$ is sandwiched between packings and coverings of projective space.

\begin{theorem}
\label{thm:covering}
Let $\delta > 0$.
Suppose there are $n$ points $x_1, \dots, x_n \in \R P^d$ such that every point $y \in \R P^d$ is strictly within distance $\delta$ from one of the~$x_j$.
Then there is an odd map $S^d_{\pi-2\delta} \to S^{n-1}$.
\end{theorem}

The following result about odd maps from metric thickenings into spheres, when combined with Theorem~\ref{thm:lower-bound}, implies Theorem~\ref{thm:main}; see Section~\ref{sec:cohom-index} for the proof. 

\begin{theorem}
\label{thm:upper-bound}
Let $X$ be a metric space with a free $\Z/2$-action, and let $\delta > 0$.
Let $V$ be an $n$-dimensional real vector space of continuous odd functions $f \colon X \to \R$, such that for every $A \subset X$ of diameter at most~$\delta$ there is an $f \in V$ such that $f|_A$ is strictly positive.
Then there is an odd map $X_\delta \to S^{n-1}$.
\end{theorem}

After proving the results outlined above, we will derive applications of Theorem~\ref{thm:main}.
Some of these are generalizations of the familiar consequences of the Borsuk--Ulam theorem, while others are new applications that require the added generality of Theorem~\ref{thm:main} to produce interesting results.
Applications of the latter kind are the following:
\vspace{2mm}
\begin{compactenum}[(i)]
\item A well-studied problem concerns generic manifold embeddings, that is, if a manifold $M$ embeds into~$\R^d$, does it embed in a particularly generic way, for example, such that tangent lines at different points are skew~\cite{ghomi2008totally} or non-parallel~\cite{harrison2020},
or such that any $k$ pairwise distinct points map to linearly independent points~\cite{blagojevic2020}.
Here we show that for embeddings of spheres into $\R^d$ a notion of genericity of embeddings is governed by the size of projective codes; see Section~\ref{sec:convexity}.
\item Simonyi and Tardos~\cite{simonyi2006local} established topological lower bounds for the circular chromatic number of a graph in terms of the topology of its box complex.
Theorem~\ref{thm:main} is already interesting in the case $d=1$ of the circle~$S^1$, where we derive the Simonyi--Tardos result as a consequence, thereby giving a new proof of their result that provides new insights into why topological bounds of the circular chromatic number depend on the parity of the coindex of the box complex; see Section~\ref{sec:chromatic}.
\end{compactenum}
\vspace{2mm}
For two of the usual consequences of the Borsuk--Ulam theorem, we prove generalizations to possibly ``overdetermined'' cases.
These consequences are:
\vspace{2mm}
\begin{compactenum}
\item[(iii)] The Ham Sandwich theorem asserts that any $d$ continuous probability measures on $\R^d$ can be simultaneously bisected by an affine hyperplane.
We prove a generalization of this result to the case of $n \ge d$ probability measures, where we have to relax the notion of bisection; see Section~\ref{sec:ham}.
\item[(iv)] In any covering of $S^n$ by $n+1$ closed sets, one of these sets contains a pair of antipodal points.
We generalize this result by showing that this remains true for coverings of $S^n$ by more than $n+1$ sets, provided that every subset of $S^n$ of a certain diameter appears in at least one of those sets; see Section~\ref{sec:covering}.
\end{compactenum}

\section{Context}

We begin with some preliminary material, definitions, and lemmas.

\subsection{Structural results for zeros of odd maps}\label{ssec:zeros}

Let $X$ and $Y$ be spaces with $\Z/2$-actions, that is, both $X$ and $Y$ are equipped with a continuous involution, which we will denote $x \mapsto -x$ for $x\in X$ and similarly $y \mapsto -y$ for $y \in Y$.
For the sphere $S^d$ and for Euclidean space $\R^d$ we always fix the antipodal $\Z/2$-action.
We will refer to any $\Z/2$-equivariant map $f \colon X\to Y$ (i.e., $f(-x) = -f(x)$ for all $x\in X$) as an \emph{odd map}.
The central result about zeros of odd maps is the Borsuk--Ulam theorem~\cite{borsuk1933}:

\begin{theorem}
Any continuous odd map $f\colon S^d \to \R^d$ has a zero.
Equivalently, there is no continuous odd map $S^d \to S^{d-1}$.
\end{theorem}

Bourgin~\cite[Theorem~4A]{bourgin1955some} and Yang~\cite{yang1954theorems} proved generalizations of this theorem asserting that for any continuous odd map $f\colon S^d \to \R^n$ with $1\le n \le d$, the preimage $f^{-1}(0)$ contains an antipodally symmetric $(d-n)$-dimensional cycle, which when mapped into the corresponding projective space $\R P^d$, is a non-zero element of $H_{d-n}(\R P^d;\Z/2)$.
From Theorem~\ref{thm:main} we can derive a Borsuk--Ulam theorem for the complementary case $n > d$, as explained in Section~\ref{sec:convexity}.
This improves on an earlier result of the authors in~\cite{ABF} that bridges applied, quantitative, and equivariant topology.

To explain the connection between Borsuk--Ulam theorems and projective codes, we say that a set $Y \subset \R^n$ is a \emph{Carath\'eodory subset} if its convex hull captures the origin in $\R^n$, i.e.\ if $0 \in \conv(Y)$.
In particular, any set $Y$ that contains antipodal points $y_0, -y_0 \in Y$, such as the image of an odd map $f\colon S^d \to \R^n$, is trivially a Carath\'eodory subset.
While the image of a continuous odd map $f\colon S^d\to\R^n$ may miss the origin, we will explain in Section~\ref{sec:convexity} how Theorem~\ref{thm:lower-bound} implies
that for some set $A \subset S^d$ of diameter strictly less than~$\pi$, the restricted image $f(A)$ is still a Carath\'eodory subset.

In order to recover the topological lower bounds for the circular chromatic number of a graph due to Simonyi and Tardos~\cite{simonyi2006local} in Section~\ref{sec:chromatic}, we will need an auxiliary result from~\cite{ABF}: Let
\[\SM_{2k}(t) = \left(\cos t, \sin t, \cos(3t), \sin(3t), \dots, \cos((2k-1)t), \sin((2k-1)t)\right)\]
be the \emph{symmetric trigonometric moment curve}.
We may think of $\SM_{2k}$ as a map $S^1 \to \R^{2k}$ under the identification $S^1=\R/2\pi\Z$, with the antipodal action specified by $t\mapsto t+\pi$.
Observe that $\SM_{2k}$ is an odd map.
Furthermore, the preimage of any Carath{\'e}odory subset of the image of $\sm_{2k}$ must have sufficiently large diameter in the circle $S^1$:

\begin{lemma}[{\cite[Theorem~5]{ABF}}]
\label{lem:sm}
Let $X \subset S^1$ be a set with diameter strictly less than~$\frac{2\pi k}{2k+1}$.
Then the convex hull of $\SM_{2k}(X)$ does not contain the origin. 
\end{lemma}

For example, in the case $k=1$, Lemma~\ref{lem:sm} reduces to the statement that no subset of $S^1\subset \R^2$ of diameter less than $\frac{2\pi}{3}$ can contain the origin in its convex hull; this is Jung's Theorem in the plane~\cite[Lemma 2]{katz1983filling}.
This lemma can be seen as a structural result for zeros of raked trigonometric polynomials; see~\cite[Theorem~4]{ABF}.
We will need the following consequence of Lemma~\ref{lem:sm} in Section~\ref{sec:chromatic}:

\begin{lemma}
\label{lem:index}
Let $k \ge 1$ be an integer and let $\delta < \frac{2\pi k}{2k+1}$.
Then there is a continuous odd map $f\colon S^1_\delta \to S^{2k-1}$.
\end{lemma}

\begin{proof}
By Lemma~\ref{lem:sm} the map $\SM_{2k} \colon S^1 \to \R^{2k}$ extends linearly to an odd map $\widehat{f} \colon S^1_\delta \to \R^{2k}\setminus \{0\}$.
Then $f(x) = \frac{\widehat{f}(x)}{\|\widehat{f}(x)\|}$ is the desired map.
\end{proof}

The bound on $\delta$ in Lemma~\ref{lem:index} is optimal: there is a homotopy equivalence $S^1_\delta \simeq S^{2k+1}$ for $\frac{2\pi k}{2k+1} \le \delta < \frac{2\pi (k+1)}{2k+3}$~\cite{moy2022vietoris,adamaszek2017vietoris}, which can furthermore be made $\Z/2$-equivariant.

\subsection{Cohomological index}
\label{sec:coh}

Observe that Theorems~\ref{thm:lower-bound}, \ref{thm:covering}, and~\ref{thm:upper-bound} are statements regarding the existence of continuous odd maps into and out of spheres. 
Results along these lines are commonly phrased in terms of the cohomological index, which we now recall. 

Let $X$ be a metric space with a free $\Z/2$-action, that is, a fixed-point free involution.
Let $S^\infty = \cup_{n=0}^\infty S^n$, where each $S^n$ is the equator of $S^{n+1}$.
Up to $\Z/2$-equivariant homotopy 
(i.e., a homotopy that at every time commutes with the $\Z/2$-action) 
there is a unique odd map $X \to S^\infty$, which gives $X / (\Z/2) \to \R P^\infty$ on quotients.
This induces a map $H^*(\R P^\infty; \F_2) \to H^*(X / (\Z/2); \F_2)$ in cohomology with $\F_2$-coefficients, where $H^*(\R P^\infty; \F_2) \cong \F_2[t]$.
The smallest integer $n$ such that $t^{n+1}$ is in the kernel of this map is the \emph{cohomological index} (or Stiefel--Whitney height) of $X$ and denoted by~$\ind(X)$; see~\cite{matousek2003using}.
If there is an odd map $X \to S^n$, then $\ind(X) \le n$.
On the other hand, if there is an odd map $S^n \to X$ then $n\le \ind(X)$.
We refer to the largest $n$ such that there is an odd map $S^n \to X$ as the \emph{coindex} of~$X$, denoted by~$\coind(X)$, which by the previous sentence satisfies $\coind(X)\le\ind(X)$.

\subsection{Metric thickenings}\label{ssec:thickenings} Let $(X,d)$ be a metric space and let $\delta \ge 0$.
The \emph{(Vietoris--Rips) metric thickening}~$X_\delta$ of $X$ at scale~$\delta$ is the set of probability measures~$\mu$ in $X$ whose support $\supp(\mu)$ is finite and has diameter at most~$\delta$, equipped with the $1$-Wasserstein metric of optimal transport.
See~\cite{AAF} for an introduction.
In particular, we can write elements $\mu \in X_\delta$ as convex combinations $\mu = \sum_{k=1}^n \lambda_ix_i$, where $\lambda_i \ge 0$, $\sum_{i=1}^n \lambda_i = 1$, and $x_1, \dots, x_n \in X$ with $d(x_i,x_j) \le \delta$ for all $1\le i,j\le n$; see Figure~\ref{fig:VRm}. 
Here we identify each point $x_i \in X$ with the Dirac measure at~$x_i$.
Notice that in this way $X$ is naturally a subspace of~$X_\delta$, embedded isometrically.

\begin{figure}[h]
\includegraphics[width=0.9in]{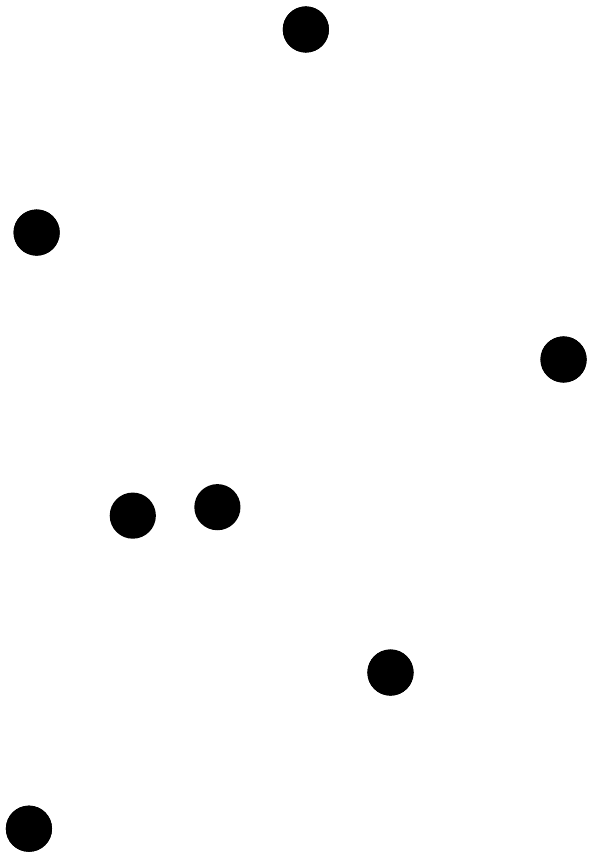}
\hspace{0.3in}
\includegraphics[width=0.9in]{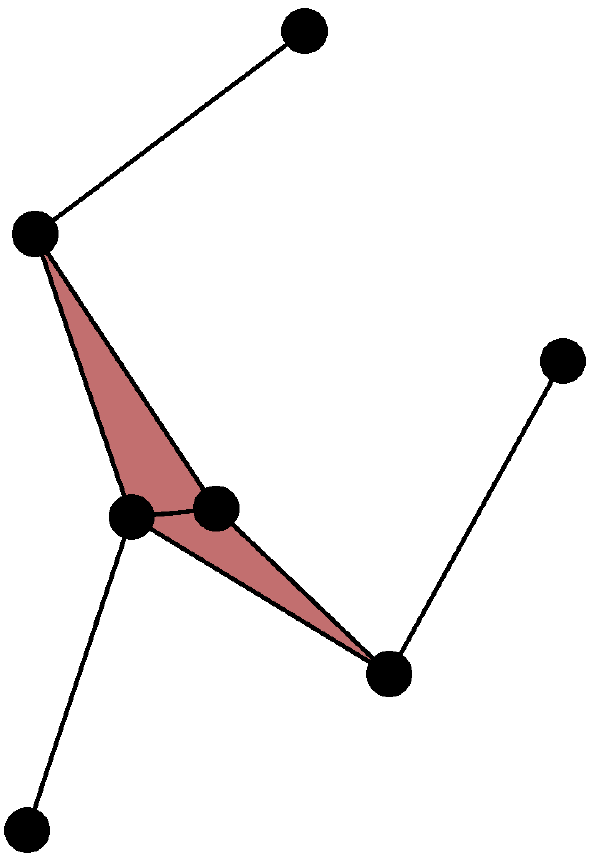}
\hspace{0.3in}
\includegraphics[width=0.9in]{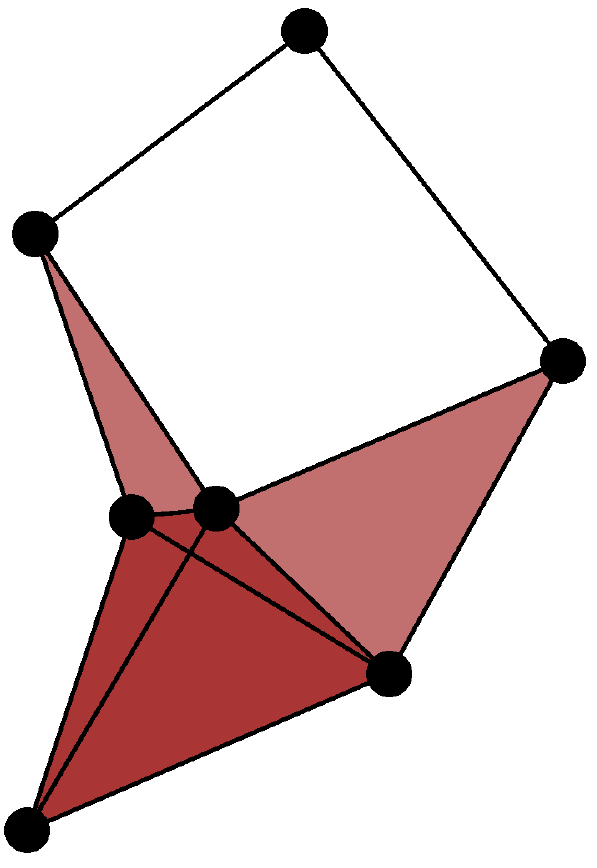}
\hspace{0.3in}
\includegraphics[width=0.9in]{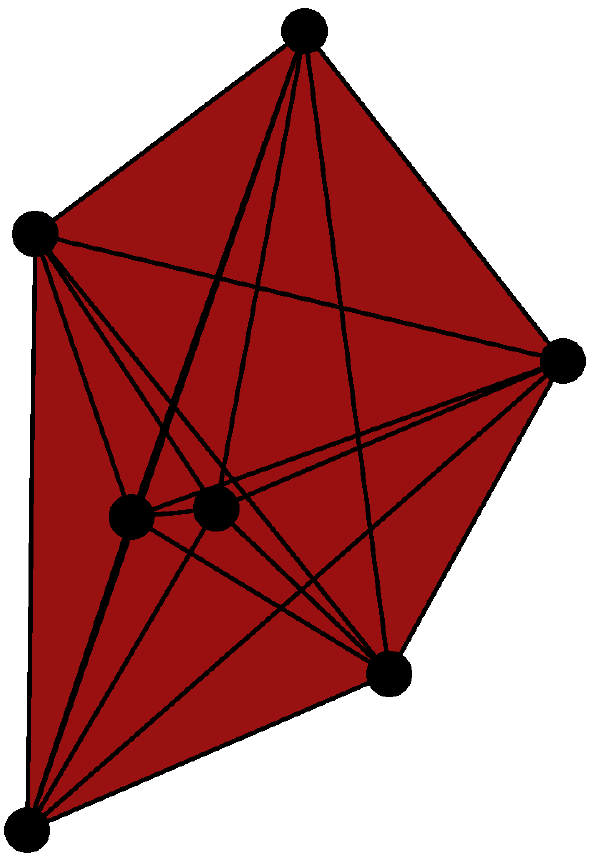}
\captionsetup{width=1\textwidth}
\caption{A metric space $X$ with 7 points, and its metric thickening $X_\delta$ at three different increasing values of $\delta>0$.
In this visualization, a Dirac delta measure supported on a point $x\in X$ is drawn as a black point, and a probability measure $\sum_{k=1}^n \lambda_ix_i$ with $d(x_i,x_j)\le \delta$ for all $i, j$ is represented by the point in the simplex $\{x_1,\ldots,x_n\}$ with barycentric coordinates $(\lambda_1,\ldots,\lambda_n)$.
}
\label{fig:VRm}
\end{figure}

Originally, metric thickenings were defined to introduce tools from metric geometry to the study of \emph{Vietoris--Rips complexes}, that is, the simplicial complex $\vr{X}{\delta}$ of all finite subsets of $X$ of diameter at most~$\delta$.
These complexes are ubiquitous in computational topology and topological data analysis~\cite{Carlsson2009}.
Despite their prevalence in these areas, the topology of Vietoris--Rips complexes and thickenings is not well-understood, unless the scale parameter $\delta$ is small~\cite{hausmann1995vietoris,Latschev2001,AAF} or the underlying metric space $X$ is simple~\cite{adamaszek2017vietoris,moy2022vietoris}.
See~\cite{HA-FF-ZV,AdamsMemoliMoyWang,MoyMasters} for relationships between Vietoris--Rips complexes and Vietoris--Rips metric thickenings in the context of applied topology.
Because metric thickenings parametrize the collection of subsets of a metric space of bounded diameter, they are the correct ``configuration space'' for studying such subsets.
Indeed, the topology of metric thickenings of spheres $S_\delta^d$ will play a key role in the proofs of our main results.

The choice to restrict the metric thickening $X_\delta$ to probability measures supported in finitely many points is essentially immaterial; this restriction is the easiest way to ensure that the $1$-Wasserstein distance between such measures is well-defined.
(Otherwise we could instead require that every probability measure on $X$ is inner regular, that is, for each open $U \subset X$ the measure $\mu(U)$ can be approximated arbitrarily closely by~$\mu(C)$, where $C \subset U$ is compact.)
Also, one could instead choose to equip the metric thickening $X_\delta$ with the $q$-Wasserstein metric for any $q\in [1,\infty)$.

Any map $f\colon X \to \R^n$ induces a natural map $f\colon X_\delta \to \R^n$ by $f(\mu) = \int_X f \ d\mu$, where for $\mu = \sum \lambda_ix_i$ the integral $\int_X f \ d\mu$ is simply the sum $\sum \lambda_if(x_i)$.
The extended map on $X_\delta$ is also denoted~$f$, since no confusion can arise: The value of $f(x)$ for $x\in X$ does not depend on whether we interpret $x$ as a point in $X$ or the Dirac measure centered at~$x$.
We call this naturally extended map $f\colon X_\delta \to \R^n$ the \emph{linear extension} of~$f$, and say that $f$ was \emph{linearly extended}.
If $f\colon X\to \R^n$ is continuous and bounded, then its linear extension $f\colon X_\delta \to \R^n$ is continuous~\cite[Lemma~5.2]{AAF}.

Suppose $X$ is equipped with a $\Z/2$-action.
Then the generator of $\Z/2$ defines an involution $\iota\colon X\to X$.
This extends to an involution $\iota\colon X_\delta \to X_\delta$ in a natural way: $\iota(\sum \lambda_ix_i) = \sum \lambda_i \iota(x_i)$.
Thus $X_\delta$ is a $\Z/2$-space as well.
If the action on $X$ is free, then the induced action on $X_\delta$ is free as well, provided that $d(x, \iota(x)) > \delta$ for all $x\in X$.

\section{The cohomological index of metric thickenings}\label{sec:cohom-index}

We will first prove lower bounds for the topology of metric thickenings in terms of the size of projective codes.
These lower bounds are in terms of the coindex of metric thickenings.
Algebraically, this is approximately captured by the (integer-valued $\Z/2$-)cohomological index; see Section~\ref{sec:coh} for a brief introduction.
For convenience, we restate these theorems from the introduction before giving the proofs; note that we now also refer to the cohomological index of a space. 

\begin{reptheorem}{thm:lower-bound}
Let $d$ and $n$ be positive integers, and let $\delta \ge \pi-p_{n+1}(\R P^d)$.
Then there is a continuous odd map $S^n \to S^d_\delta$; in particular, $n \le \coind(S^d_\delta) \le \ind(S^d_\delta)$.
\end{reptheorem}

\begin{proof}
We will construct an odd map $f\colon S^n \to S^d_\delta$.
Let $\alpha = p_{n+1}(\R P^d)$.
By compactness, we can find $x_1, \dots, x_{n+1} \in \R P^d$ that are pairwise at distance at least~$\alpha$.
Each $x_j$ lifts to two points $x_j^-$ and $x_j^+$ in~$S^d$.
These points are antipodes at distance~$\pi$, but the distance to any other point $x_i^\pm$ for $i \ne j$ is at most $\pi-\alpha$.
Indeed, the point $x_i^\pm$ lies on a shortest path connecting $x_j^+$ to~$x_j^-$.
Since its distance to $x_j^-$ is at least~$\alpha$, it is within a distance of $\pi-\alpha$ from~$x_j^+$.
Symmetrically, since its distance to $x_j^+$ is at least $\alpha$, it is within a distance of $\pi-\alpha$ from~$x_j^-$. 
	
We will think of the domain of the odd map $f\colon S^n \to S^d_\delta$ as the $\ell_1$ unit sphere in~$\R^{n+1}$, that is, as the boundary of the $(n+1)$-dimensional crosspolytope with vertices $\pm e_1, \dots, \pm e_{n+1}$.
Define $f$ on vertices by mapping $\pm e_j$ to the Dirac measure at~$x_j^\pm$.
Then extend linearly, that is, given a point $\sum_{i=1}^{n+1} \lambda_i \varepsilon_i e_i$ in the $\ell_1$ unit sphere for some choice of signs $\varepsilon_i \in \{-1,1\}$ and coefficients $\lambda_i \ge 0$ with $\sum_{i=1}^{n+1} \lambda_i = 1$, define $f(\sum_{i=1}^{n+1} \lambda_i \varepsilon_i e_i) = \sum_{i=1}^{n+1} \lambda_i \varepsilon_i f(e_i)$.
This is a measure supported on at most $n+1$ points whose pairwise distances are at most $\pi-\alpha \le \delta$.
\end{proof}

The proof above shows that a projective code $x_1, \dots, x_{n+1} \in \R P^d$, where any two points are at distance at least~$\alpha$, induces a linear map from the boundary of the $(n+1)$-dimensional crosspolytope $S^n \to S^d_\delta$ for $\delta \ge \pi-\alpha$.
Here $S^d_\delta$ is a subset of the vector space of signed measures on~$S^d$, so linearity is well-defined.
Conversely, any map $S^n \to S^d_\delta$ that is linear in this sense induces a projective code, where elements are at distance at least~$\pi-\delta$.
The cohomological index of $S^d_\delta$ provides an algebraic witness for the existence of odd maps $S^n \to S^d_\delta$, and thus the parametrized family of cohomological indices of $S^d_\delta$ as $\delta$ ranges over $(0, \pi)$ is a topological version of projective codes.

As remarked in the introduction,  coverings of projective space by metric balls dually provide upper bounds for the topology of metric thickenings; see Theorem~\ref{thm:covering}.
In other words, simple volume bounds for the size of projective codes apply more generally to bound the topology of metric thickenings.
A collection of $n$ metric $\delta$-balls in $\R P^d$ cannot be pairwise disjoint if the sum of their volumes exceeds the volume of~$\R P^d$.
This bound for projective codes applies more generally as an upper bound for the topology of metric thickenings of spheres.

\begin{reptheorem}{thm:covering}
Let $\delta > 0$.
Suppose there are $n$ points $x_1, \dots, x_n \in \R P^d$ such that every point $y \in \R P^d$ is strictly within distance $\delta$ from one of the~$x_j$.
Then there is an odd map $S^d_{\pi-2\delta} \to S^{n-1}$; in particular, $\ind(S^d_{\pi-2\delta}) \le n-1$.
\end{reptheorem}

\begin{proof}
We think of each $x_j \in \R P^d$ as a pair of antipodal points $\pm x_j$ in~$S^d$.
If $y \in S^d$ has distance less than $\delta$ from $x_j$ and $z \in S^d$ has distance less than $\delta$ from~$-x_j$, then $y$ and $z$ are at a distance exceeding $\pi-2\delta$.
Thus for any $\mu \in S^d_{\pi-2\delta}$ the support of $\mu$ does not intersect both $B_\delta(x_j)$ and~$B_\delta(-x_j)$.
Let $f_j \colon S^d_{\pi-2\delta} \to \R$ be the map that for $\mu \in S^d_{\pi-2\delta}$ with support disjoint from $B_\delta(-x_j)$ is defined as
\[\inf\{d(\mu,\nu)~:~\nu \in S^d_{\pi}, \  \supp(\nu) \cap B_\delta(x_j) = \emptyset\},\]
where the distance between measures is the 1-Wasserstein distance.
Extend this map to all of $S^d_{\pi-2\delta}$ by declaring that $f_j(-\mu) = -f_j(\mu)$.
It is easily verified that this is a well-defined, continuous odd map.
The map $f\colon S^d_{\pi-2\delta} \to \R^n$ given by $f = (f_1, \dots, f_n)$ is odd and misses the origin.
Thus $\mu \mapsto \frac{f(\mu)}{\|f(\mu)\|}$ is an odd map $S^d_{\pi-2\delta} \to S^{n-1}$.
\end{proof}

The next theorem, when combined with Theorem~\ref{thm:lower-bound}, will allow us to prove our main result, Theorem~\ref{thm:main}.

\begin{reptheorem}{thm:upper-bound}
Let $X$ be a metric space with a free $\Z/2$-action, and let $\delta > 0$.
Let $V$ be an $n$-dimensional real vector space of continuous odd functions $f \colon X \to \R$, such that for every $A \subset X$ of diameter at most~$\delta$ there is an $f \in V$ such that $f|_A$ is strictly positive.
Then there is an odd map $X_\delta \to S^{n-1}$; in particular, $\ind(X_\delta) \le n-1$.
\end{reptheorem}

\begin{proof}
Fix a basis $B$ of~$V$, so $|B|=n$.
For $\mu \in X_\delta$ and $f \in V$ let $f(\mu) = \int_X f \ d\mu$.
Define $\Phi \colon X_\delta \to \R^n$ by $\Phi(\mu) = (f(\mu))_{f\in B}$.
Every element of $V$ is a linear combination of $B$ and thus a function of the form $x \mapsto \langle z, \Phi(x) \rangle$ for some $z\in \R^n$.
Here $x$ denotes the Dirac measure~$\delta_x$.
Let $\mu \in X_\delta$ be arbitrary, and denote its support by~$A$.
There is some $f \in V$ such that $f|_A$ is strictly positive.
Thus $\langle z, \Phi(x) \rangle > 0$ for some $z \in \R^n$ and every $x \in A$.
In particular, $\langle z, \Phi(\mu)\rangle > 0$ and thus $\Phi$ maps to $\R^n \setminus \{0\} \simeq S^{n-1}$.
\end{proof}

\begin{proof}[Proof of Theorem~\ref{thm:main}]
Let $V$ be an $n$-dimensional vector space of continuous odd maps $S^d \to \R$.
We must show that there is a set $A \subset S^d$ of diameter at most $\delta:=\pi-p_{n+1}(\R P^d)$ such that for every $f \in V$ the restriction $f|_A$ is not strictly positive.

Theorem~\ref{thm:lower-bound} implies $\ind(S^d_\delta) \ge n$.
If there were no such set $A\subset S^d$, then for all sets of diameter at most $\delta$ there would be some $f\in V$ such that the restriction $f|_A$ is strictly positive, so Theorem~\ref{thm:upper-bound} would imply $\ind(S^d_\delta) \le n-1$, a contradiction.
\end{proof}

\begin{remark}\label{rmk:gen-BU}
Instead of $\delta=\pi-p_{n+1}(\R P^d)$ we could have phrased Theorem~\ref{thm:main} for the smallest $\delta$ such that the metric thickening $S^d_\delta$ has cohomological index at least~$n$, i.e., such that $\ind(S^d_\delta) \ge n$.
For $n = d$, this $\delta$ is zero, since $S^d$ is a subspace of $S^d_\delta$ for any $\delta \ge 0$.
In the case $n=d$ and $\delta =0$, this stronger version of Theorem~\ref{thm:main} says that given any $d$-dimensional vector space $V$ of continuous odd maps $S^d\to \R$, there is a point $x\in S^d$ (i.e., a set of diameter zero) such that $f(x)=0$ (since $f\in V$ implies $-f\in V$) for all $f\in V$.
By choosing $V$ to be generated by the $d$ coordinate functions of an odd map $g\colon S^d\to \R^d$, we hence obtain the standard Borsuk--Ulam Theorem stating there exists some $x\in S^d$ with $g(x)=0$.
\end{remark}

To conclude this section, we give an example of explicit bounds that may be derived from Theorem~\ref{thm:lower-bound}.
For general upper and lower bounds for the size of projective codes, see Levenshtein~\cite{levenshtein1998}.
For our explicit example, we are interested in projective codes whose size exceeds the dimension by a fixed integer.
Bukh and Cox~\cite{bukh2020} study the behaviour of $p_{d+k}(\R P^d)$ for fixed $k$ as $d$ grows.
Among other results they show:

\begin{theorem}[Bukh and Cox~\cite{bukh2020}, Theorem~21] 
\label{thm:bc}
Let $k \ge 1$ be an integer.
Then 
\[p_{d-1+k}(\R P^{d-1}) \ge \arccos\left((1+o(1))\frac{2\sqrt{k+1}}{d}\right),\]
where $o(1)$ denotes a function that converges to $0$ as $d\to \infty$.
\end{theorem}

By combining Theorem~\ref{thm:bc} with Theorem~\ref{thm:lower-bound}, we immediately get the following consequence for the topology of Vietoris--Rips metric thickenings of spheres:

\begin{corollary}\label{cor:high-coindex}
Let $k\ge 1$ be an integer and $\delta > \frac{\pi}{2}$.
Then for $d = d(k, \delta)$ sufficiently large, there is a continuous odd map $S^{d+k} \to S^d_\delta$.
Indeed, we need to choose $d$ sufficiently large such that $\delta > \pi - \arccos\left((1+o(1))\frac{2\sqrt{k+2}}{d+1}\right)$, where $o(1)$ denotes a function that converges to $0$ as $d\to \infty$.
\end{corollary}

\begin{proof}
By Theorem~\ref{thm:bc} we have $p_{d+k+1}(\R P^d)\ge \arccos\left((1+o(1))\frac{2\sqrt{k+2}}{d+1}\right)\to \frac{\pi}{2}$ as $d\to\infty$.
So for $\delta>\frac{\pi}{2}$ we have $\delta\ge\pi-p_{d+k+1}(\R P^d)$ for $d$ sufficiently large, in which case Theorem~\ref{thm:lower-bound} gives an odd map $S^{d+k}\to S^d_\delta$.
\end{proof}

For $\delta \le \frac{\pi}{2}$ we have that $S^d_\delta \simeq S^d$~\cite{AAF}.
Corollary~\ref{cor:high-coindex} asserts that for any $\delta > \frac{\pi}{2}$, there is an arbitrarily large jump in coindex between $S^d_{\pi/2}$ and $S^d_\delta$ for $d$ sufficiently large.

\section{Convexity properties of odd maps and generic manifold embeddings}
\label{sec:convexity}

The fundamental result about zeros of odd maps is the classical Borsuk--Ulam theorem: Any continuous odd map $S^d \to \R^d$ has a zero, or equivalently, any continuous map $f\colon S^d \to \R^d$ identifies two antipodal points, $f(-x) = f(x)$.
Generalizations of this result that quantify the size of the zero set or preimage of a point for maps $S^d \to \R^n$ for $n \le d$ were proven by
Bourgin~\cite{bourgin1955some}, Yang~\cite{yang1954theorems,yang1955theorems}, Almgren~\cite{almgren1972theory}, and Gromov~\cite{gromov1983filling,gromov2003isoperimetry,guth2008waist,memarian2011gromov}.
Our Theorem~\ref{thm:main} studies an approximate Borsuk--Ulam result for maps $S^d \to \R^n$ in which we instead have $n \ge d$.

To explain this connection, recall from Section~\ref{ssec:zeros} that we say a set $Y \subset \R^n$ is a \emph{Carath\'eodory subset} if its convex hull captures the origin, $0 \in \conv(Y)$.
Any set that contains antipodal points, such as the image of an odd map $S^d\to\R^n$, is trivially a Carath\'eodory subset.
We will explain how Theorem~\ref{thm:lower-bound} shows
that for some set $A \subset S^d$ of diameter strictly less than~$\pi$ the restricted image $f(A)$ still is a Carath\'eodory subset.
For a compact metric space~$X$ with free $\Z/2$-action, we denote by $c_n(X)$ the infimum over all $t\in [0,\infty)$ such that, for any odd map $f \colon X \to \R^n$, there is a set $A \subset X$ with diameter bound $\diam(A) \le t$ and with $0 \in \conv(f(A))$.
We call $c_n(X)$ the \emph{$\Z/2$-Carath{\'e}odory number} of~$X$.
The Borsuk--Ulam theorem may be phrased as $c_d(S^d) = 0$.
Upper bounds for $c_n(S^d)$, for $n > d$, follow easily from Theorem~\ref{thm:lower-bound}:

\begin{corollary}
\label{cor:main}
The $\Z/2$-Carath\'eodory number of the $d$-sphere satisfies the inequality
\[c_n(S^d) \le \pi - p_{n+1}(\R P^d).\]
\end{corollary}

\begin{proof}
Let $f \colon S^d \to \R^n$ be a continuous odd map, and let $\delta = \pi - p_{n+1}(\R P^d)$.
The map $f$ has a canonical continuous odd extension $F\colon S^d_\delta \to \R^n$ defined by $F(\mu) = \int_{S^d} f \ d\mu$.
In addition, Theorem~\ref{thm:lower-bound} gives a continuous odd map $S^n \to S^d_\delta$.
By the Borsuk--Ulam theorem applied to the odd composition $S^n \to S^d_\delta \to \R^n$, there exists some $\mu \in S^d_\delta$ with $0 = \int_{S^d} f \ d\mu$.
Thus $0 \in \conv\{\mathrm{supp}(\mu)\}$ with $\diam(\supp(\mu))\le\delta$.
\end{proof}

\begin{remark}
\label{rem:explicit-code-bounds}
Corollary~\ref{cor:main} substantially improves a result of the present authors in~\cite{ABF} that shows $c_{2kd+2d-1}(S^{2d-1}) \le \frac{2\pi k}{2k+1}$ (with equality in the case $d=1$).
For example, consider the projective code given by lines through the vertices of the hypercube: Let $C \subset \R^d$ be the set of vectors whose coordinates are all $\pm 1$.
This is an antipodal set, $C = -C$, determining $2^{d-1}$ lines through the origin.
The angle between distinct lines is at least~$\arccos(1-\frac{1}{d})$.
More generally,  fix $d,n\ge 1$, and let 
\[S_r = \{x \in \Z^d~:~\|x\|^2 = r \text{ and } |x_i| \le n \text{ for all } i\}.\]
Note $\{x \in \Z^d~:~|x_i| \le n \text{ for all } i\}$ has cardinality $(2n+1)^d$, that $S_0$ has cardinality $1$, and that $S_r = \emptyset$ for $r > dn^2$.
Hence some $S_r$ with $1\le r\le dn^2$ has cardinality at least $\frac{(2n+1)^{d}-1}{dn^2} \ge \frac{(2n)^d}{dn^2} = \frac{2^d}{d}\cdot\frac{n^d}{n^2} > n^{d-2}$.
For $x, y \in S_r$ with $x \ne \pm y$ the absolute value of the inner product $|\langle x,y \rangle|$ is an integer and less than~$r$, in particular $|\langle x,y \rangle| \le r-1$.
Thus the angle between distinct lines determined by points in~$S_r$ is at least $\arccos(1-\frac{1}{r})$.
It follows that $p_{n^{d-2}+1}(\R P^{d-1}) \ge \arccos(1-\frac{1}{dn^2})$, and hence
\[c_{n^{d-2}}(S^{d-1}) \le \pi - \arccos\left(1-\frac{1}{dn^2}\right).\]
Since $\arccos(1-x)$ converges to $0$ as $x\to 0$ in a linear fashion, this shows that $\pi-c_{n^{d-2}}(S^{d-1}) = \Omega(\frac{1}{dn^2})$ as $n \to \infty$.
By contrast, the earlier bound from~\cite{ABF} would only give $\pi-c_{n^{d-2}}(S^{d-1}) = \Omega(\frac{1}{n^{d-2}})$ as $n \to \infty$.
For a concrete example, recall that the $600$-cell (also known as the hypericosahedron or hexacosichoron) has $120$ centrally-symmetric vertices, which determine $60$ lines through the origin in $\R^4$ with minimal angle~$\frac{\pi}{5}$.
Thus $c_{59}(S^3) \le \pi-\frac{\pi}{5}$ by Corollary~\ref{cor:main}, whereas our earlier bound (with $d=2$ and $k=14$) only gives $c_{59}(S^3) \le \pi -\frac{\pi}{29}$.
\end{remark}

\begin{remark}
Instead of viewing Corollary~\ref{cor:main} as an upper bound for $c_n(S^d)$ in terms of~$p_{n+1}(\R P^d)$, equivalently $p_{n+1}(\R P^d) \le \pi - c_n(S^d)$ gives an upper bound for projective codes in terms of zeros of odd maps.
It remains an interesting open problem whether upper bounds for projective codes may be improved using this framework by exhibiting odd maps such that the image of sets up to a certain diameter may be separated from the origin.
On the other hand, Theorem~\ref {thm:covering} shows that volume bounds for projective codes more generally provide upper bounds for the cohomological index of metric thickenings.
\end{remark}

In much the same way that the Borsuk--Ulam theorem prohibits an embedding of $S^d$ into~$\R^d$, Corollary~\ref{cor:main} gives bounds for how generic an embedding $S^d \hookrightarrow \R^n$, $n > d$, may be.
Here the convex hulls of images of nearby points must intersect, where ``nearby'' is quantified by the distance of projective codes of size~$n+1$.

\begin{corollary}
\label{cor:antipodal}
Let $d$ and $n$ be positive integers, and let $\delta \ge c_n(S^d)$, which is satisfied if $\delta \ge \pi-p_{n+1}(\R P^d)$.
Then for any map $f\colon S^d \to \R^n$ there is a set $X \subset S^d$ of diameter at most~$\delta$ such that the convex hulls of $f(X)$ and $f(-X)$ intersect.
Moreover, the preimages of the intersection point have the same coefficients with respect to $X$ and~$-X$, that is, $\int_{S^d} f(x) \ d\mu = \int_{S^d} f(-x) \ d\mu$ for some $\mu \in S^d_\delta$.
\end{corollary}

By Carath{\'e}odory's theorem, we may take $X$ to have size at most $n+1$.

\begin{proof}
Let $\delta\ge c_n(S^d)$, which for example holds if $\delta\ge \pi-p_{n+1}(\R P^d)$ by Corollary~\ref{cor:main}.
Define $F\colon S^d_\delta \to \R^n$ by $F(\mu) = \int_{S^d} (f(x) - f(-x)) \ d\mu(x)$.
The map $F$ is odd, and thus has a zero by the definition of the $\Z/2$-Carath\'eodory number $c_n(S^d)$.
This implies $\int_{S^d} f(x) \ d\mu(x) = \int_{S^d} f(-x) \ d\mu(x)$ for some probability measure $\mu$ on $S^d$ whose support has diameter at most~$\delta$.
The point $\int_{S^d} f(x) \ d\mu(x)$ lies in the convex hull of $f(X)$, where $X$ denotes the support~$\mu$, while $\int_{S^d} f(-x) \ d\mu(x) \in \conv \ f(-X)$.
\end{proof}

\begin{example}
Let $\alpha$ be the common angle between the six (equiangular) lines in $\R^3$ determined by the vertices of the icosahedron.
That is, let $\phi = \frac{1+\sqrt{5}}{2}$ be the golden ratio, and let $\alpha = \arccos(\frac{\phi}{1+\phi^2})$.
Corollary~\ref{cor:antipodal} shows that for any map $f \colon S^2 \to \R^5$ there is a set $X\subset S^2$ of diameter at most $\pi-\alpha$ such that $\conv \ f(X) \cap \conv \  f(-X) \ne \emptyset$.
\end{example}

\section{Topological bounds for the circular chromatic number}
\label{sec:chromatic}

Lov\'asz's proof of Kneser's conjecture~\cite{lovasz1978} introduced topological lower bounds to the study of chromatic numbers of graphs.
This classical result has been extended and adapted in various ways.
For example, Simonyi and Tardos~\cite{simonyi2006local} proved topological lower bounds for the circular chromatic number of a graph.

\begin{definition}\label{def:circular-coloring}
Given a graph $G$ with vertex set $V$ and edge set~$E$, an $r$\emph{-circular coloring of }$G$ is a map $f\colon V\to\R P^1$ such that $d(f(v),f(w))\geq \frac{\pi}{r}$ whenever $(v,w)\in E$.
Recall that $\R P^1$ simply denotes the circle of circumference~$\pi$.
\end{definition}

\begin{definition}\label{def:circular-chromatic}
The \emph{circular chromatic number of a graph} $G$ is 
\[\chi_c(G):=\inf\{r>0\mid G\text{ is }r\text{-circular colorable}\}.\]
\end{definition}

The infimum in the definition of the circular chromatic number of a graph $G$ is known to be attained if $G$ is finite.
The topological lower bounds for $\chi_c(G)$ are in terms of the topology of the box complex, which we define now.
Let $G$ be a graph on vertex set $\{1,2,\dots,n\}$.
The box complex $B_0(G)$ of $G$ is a simplicial complex on vertex set $Q = \{1,\dots,n\} \cup \{-1,\dots,-n\}$.
A subset $\sigma \subset Q$ is a face of $B_0(G)$ if for every $v \in \sigma$ with $v > 0$ and every $w \in \sigma$ with $w < 0$, the tuple $(v,-w)$ forms an edge of~$G$.
In particular, $\{1,\dots,n\}$ and $\{-1,\dots,-n\}$ are faces of~$B_0(G)$.
The complex $B_0(G)$ has a free $\Z/2$-action which swaps the vertices $v$ and $-v$ and extends linearly to the faces of~$B_0(G)$.
The box complex $B_0(G)$ that we use is only one of several possible definitions; see for example~\cite{csorba2007homotopy,matousek2002topological,simonyi2006local} which study both the box complex $B_0(G)$ along with variants thereof.

We can now state the result of Simonyi and Tardos:

\begin{theorem}[Simonyi--Tardos~{\cite[Theorem 6]{simonyi2006local}}]
\label{thm:Simonyi-Tardos}
Given a finite graph $G$, if there
is a continuous odd map $S^{2k-1} \to B_0(G)$, then $\chi_c(G) \ge 2k$.
\end{theorem}

Here we show that Theorem~\ref{thm:Simonyi-Tardos} is a consequence of bounds for the topology of metric thickenings of the circle.
Our proof in particular gives an explanation for the dependence of the bounds for the circular chromatic number on the parity of the coindex of $B_0(G)$: When $\delta$ exceeds a critical value when the index of the metric thickenings of the circle $S^1_\delta$ increases, the index in fact increases by two (see Lemma~\ref{lem:index}).
Our proof of Theorem~\ref{thm:Simonyi-Tardos} is based on Theorem~\ref{thm:chromatic-cone} and Lemmas~\ref{lem:chromatic-cone}--\ref{lem:coindex-cone}, which we now provide.

\begin{theorem}\label{thm:chromatic-cone}
Let $G$ be a finite graph with vertex set $V(G)$ and edge set $E(G)$, and let $CG$ denote the cone over $G$ with vertex set $V(CG)=
V(G)\cup \{\bullet \}$ and edge set $E(CG)=E(G)\cup \{(\bullet,v)\mid v\in V(G)\}$.
If there is a continuous odd map $S^{2k} \to B_0(CG)$, then $\chi_c(CG)\geq 2k+1$.
\end{theorem}

\begin{proof}
Suppose $G$ is a finite graph with $\chi_c(CG)< 2k+1$.
Let the map $f\colon V(CG)\to\R P^1$ be given such that for some small $\varepsilon > 0$ we have $d_{\R P^1}(f(v),f(w))\ge\frac{\pi}{2k+1}+\varepsilon$ whenever $(v,w)\in E(CG)$.
Think of $\R P^1$ as $[0,\pi]$ with $0$ and $\pi$ identified into one point.
Thus, $f$ induces a map $V(CG) \to [0,\pi)$, and we choose $f$ so that $f(\bullet)=0$.
Now, thinking of $S^1$ as $[0,2\pi]$ with $0$ and $2\pi$ identified, $f$ induces a map $V(CG) \to S^1$ that maps all vertices to the upper semi-circle.
By antipodal symmetry we can extend this to a map $F\colon (V(CG)\cup -V(CG))\to S^1$ with $F(\bullet)=0$ and $F(-\bullet)=\pi$.
In particular, if we compose $F$ with the double-covering $S^1\to \R P^1$, then both $v$ and $-v$ map to~$f(v)$.

We will prove that $F$ extends to a well-defined odd map $\widetilde{F}\colon B_0(CG)\to S^1_{\delta}$, where $\delta = \frac{2\pi k}{2k+1} - \varepsilon$.
There are three cases to check.
First, for this we need to show that for any face $\sigma$ of~$B_0(CG)$ the set $\{F(v) \ : \ v \in \sigma\} \subset S^1$ has diameter at most~$\delta$.
For any $v \in V(G)$ we have $d_{\R P^1}(f(v),f(\bullet)) \ge \frac{\pi}{2k+1}+\varepsilon$, and thus $F(v)$ is at distance at least $\frac{\pi}{2k+1}+\varepsilon$ from both $F(\bullet)$ and $-F(\bullet) = F(-\bullet)$ in $S^1$.
In particular,
\begin{equation}
\label{eq:dS1}
d_{S^1}(F(v), F(\bullet)) = \pi-d_{S^1}(F(v),-F(\bullet)) \le \pi-\left(\tfrac{\pi}{2k+1}+\varepsilon\right)=\delta,
\end{equation}
and similarly $d_{S^1}(F(-v), F(-\bullet)) \le \delta$.
Second, since $F(\bullet)=0$, note~\eqref{eq:dS1} implies that the set of points $\{F(v)~:~v\in V(G)\}$ is contained in an interval of length $\delta$ in the upper semi-circle of $S^1$.
Thus for any $v,w \in V(G)$ we have that $d_{S^1}(F(v), F(w)) \le \delta$ and $d_{S^1}(F(-v), F(-w)) \le \delta$.
Third, if $(v,w) \in E(G)$, then $d_{\R P^1}(f(v), f(w)) \ge \frac{\pi}{2k+1}+\varepsilon$, and so $d_{S^1}(F(v), F(-w)) = d_{S^1}(F(-v), F(w)) \le \delta$.

The map $F$ sends the vertices of $B_0(CG)$ to points in~$S^1$, such that vertices that span a face of $B_0(CG)$ map to points at pairwise distance at most~$\delta$.
Let $x$ be a point contained in the interior of some face $\sigma$ of~$B_0(CG)$, say $x = \sum_{i=1}^\ell \lambda_iv_i$ for $v_1,\dots, v_\ell$ vertices of $\sigma$ and $\lambda_i > 0$ with $\sum \lambda_i = 1$.
Then thinking of $F(v_i)$ as the Dirac measure based at~$F(v_i)$, we can define $\widetilde{F}(x) = \sum \lambda_i F(v_i)$, which is an element of~$S^1_\delta$.
This completes the definition of the continuous odd map $\widetilde{F}\colon B_0(CG)\to S^1_{\delta}$.
By Lemma~\ref{lem:index} there is a continuous odd map $S^1_{\delta} \to S^{2k-1}$, and by composition we get a continuous odd map $B_0(CG)\to S^{2k-1}$.
Hence their cannot be a continuous odd map $S^{2k}\to B_0(CG)$, as their composition $S^{2k} \to S^{2k-1}$ would contradict the Borsuk--Ulam theorem.
\end{proof}

\begin{lemma}\label{lem:chromatic-cone}
For a finite graph $G$, $\chi_c(G) \ge \chi_c(CG)-1$.
\end{lemma}

\begin{proof}
Let $f \colon V(G) \to \R P^1$ be a circular $r$-coloring of~$G$.
We need to construct a circular $(r+1)$-coloring of~$CG$.
Parametrize $\R P^1$ as $[0,\pi]$ with $0$ and $\pi$ identified into one point, which we will denote by~$x_0$.
Choose the parametrization in such a way that some vertex $v \in V(G)$ satisfies $f(v) = x_0$.
    
Now think of $f$ as a map $f\colon V(G) \to [0, \pi)$.
Define a map $g\colon V(CG) \to [0, \pi)$ by $g(v) = \frac{r}{r+1}f(v)$ for $v \in V(G)$ and $g(\bullet) = \frac{r}{r+1}\pi$.
It is easily verified that $g$ is a circular $(r+1)$-coloring of~$CG$.
\end{proof}

\begin{lemma}\label{lem:coindex-cone}
For a finite graph $G$, $\coind(B_0(CG)) \ge \coind(B_0(G))+1$.
\end{lemma}

\begin{proof}
To prove the lemma, it is sufficient to prove that $B_0(CG)$ is the suspension of $B_0(G)$ (see, for instance,~\cite[Lemma~3.1]{simonyi2006local}).
Let $G$ have vertex set $\{1,2,\dots,n\}$, and so $B_0(G)$ has vertex set $\{1,\dots,n\} \cup \{-1,\dots,-n\}$.
Note that the box complex $B_0(CG)$ has two additional vertices that are not present in $B_0(G)$, which we will denote by $\bullet$ and $-\bullet$. 
To prove that $B_0(CG)$ is the suspension of $B_0(G)$, we must verify that both $\sigma\cup\{\bullet\}$ and $\sigma\cup \{-\bullet\}$ are faces of $B_0(CG)$ for any given face $\sigma$ of $B_0(G)$.
We do only the former, as the argument for the latter is nearly identical.
Toward that end, for a given face $\sigma\in B_0(G)$, consider any $v\in \sigma\cup\{\bullet\}$ with $v> 0 $ and any $w\in \sigma\cup\{\bullet\}$ with $w<0$ (which implies $w\neq \bullet$).
There are two possibilities: If $v\in\sigma$, we have $(v,-w)\in E(G)\subset E(CG)$ because $\sigma\in B_0(G)$. 
Otherwise, if $v=\bullet$, we have $(v,-w)=(\bullet,-w)\in E(CG)$ by the definition of $CG$. 
Hence $\sigma\cup \{\bullet\}$ is a face of $B_0(CG)$.
\end{proof}

We are prepared to prove Theorem~\ref{thm:Simonyi-Tardos}.

\begin{proof}[Proof of Theorem~\ref{thm:Simonyi-Tardos}]
Since there is a continuous odd map $S^{2k-1} \to B_0(G)$, we have $\coind(B_0(G))\ge 2k-1$.
So $\coind(B_0(CG))\geq 2k$ by Lemma~\ref{lem:coindex-cone}, i.e.\ there is a continuous odd map $S^{2k}\to B_0(CG)$.
Hence, Theorem~\ref{thm:chromatic-cone} implies $\chi_c(CG)\geq 2k+1$.
Finally, by Lemma~\ref{lem:chromatic-cone} we have 
$\chi_c(G) \ge \chi_c(CG)-1 \ge (2k+1)-1 = 2k$.
\end{proof}

\section{Generalizations of the Ham Sandwich theorem}\label{sec:ham}

The Ham Sandwich theorem, originally conjectured by Steinhaus in the ``Scottish book,'' was proved by Banach and Stone--Tukey~\cite{stone1942generalized}.
See~\cite{beyer2004early} for the history of this result, including a note on Banach's contribution to the proof in dimension $3$.

\begin{theorem}[Ham Sandwich theorem]\label{thm:ham-sandwich}
Let $\mu_1,\dots,\mu_d$ be finite Borel measures in $\R^d$ such that every hyperplane has measure $0$ for each of the~$\mu_i$.
Then there exists a hyperplane~$H$ such that $\mu_i(H^+)=\frac{1}{2}\mu_i(\R^d)$ for $i=1,\dots, d$, where $H^+$ denotes one of the half-spaces defined by~$H$.
\end{theorem}

Informally, in the case of three measurable subsets of $\R^3$, a sandwich made of three ingredients may be equipartitioned (that is, divided into two sandwiches each containing equal amounts of all three ingredients) using a single planar slice. 

The Ham Sandwich theorem is perhaps the most well-known example of a mass partition theorem. 
In general, such theorems seek to establish the existence of families of partitions of Euclidean spaces splitting a given family of measures in a particular way.
For example, Gr{\"u}nbaum~\cite{grunbaum1960partitions} asked if it is always possible to divide a single mass in $\R^d$ by $d$ hyperplanes to obtain $2^d$ equal parts. 
This question was answered in the affirmative by Hadwiger~\cite{hadwiger1966simultane} for $d\leq 3$, but was proved to be false for $d\geq 5$ by Avis~\cite{avis1984non}.
See~\cite{roldan2022survey,de2019discrete} for surveys of mass partition theorems and related results in combinatorial topology and combinatorial geometry, including Ham Sandwich-type theorems allowing multiple planar slices and theorems concerning the fair division of cakes and necklaces.
There are versions of the Ham Sandwich theorem studying when one can bisect more masses than the dimension of the ambient space, for example including divisions by algebraic surfaces~\cite{stone1942generalized} or divisions of projections onto hyperplanes~\cite{schnider2020ham}.

Theorem~\ref{thm:main} and Corollary~\ref{cor:antipodal} allow us to extend the Ham Sandwich theorem to the setting in which the number of measures exceeds the dimension of the ambient space.
For the purpose of illustration, we refer to the following as the ``log bundle'' theorem.

Informally, in the case of three measurable subsets of the disk, Theorem~\ref{thm:logs} says the following:
Suppose a bundle of three logs needs to be divided up using equipment that permits only two kinds of cuts.
First, we can perform ``horizontal cuts,'' perpendicular to the bundle (Figure~\ref{fig:logs} (left)), dividing the bundle up into shorter bundles.
Secondly, through the centers of each resulting shorter bundles, we can perform a single ``vertical cut'' to produce two hemi-bundles (Figure~\ref{fig:logs} (right)).
Furthermore,
the blade used to perform each ``vertical cut'' is on a fixed pivot that cannot swivel by an angle of more than $\frac{2\pi}{3}$.
Then, in light of these constraints, it is always possible to obtain an equipartition of the bundle  (that is, two piles of wood each containing exactly half the mass of each log) by selecting exactly one of each of the resulting hemi-bundles from each of the shorter bundles.

\begin{figure}[h]
\centering
\begin{minipage}{.4\textwidth}
  \centering
  \includegraphics[width=.5\linewidth]{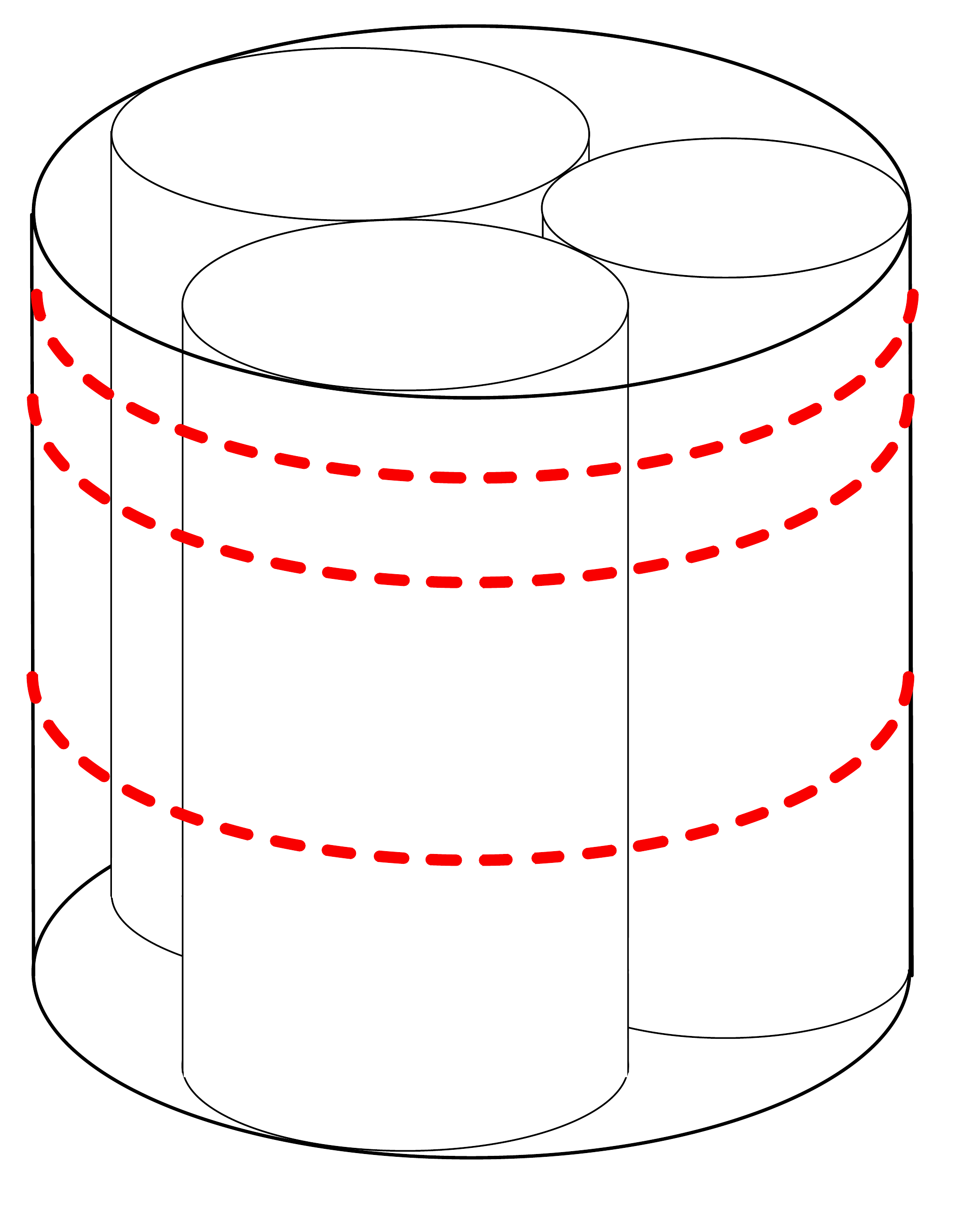}
\end{minipage}%
\begin{minipage}{.4\textwidth}
  \centering
  \includegraphics[width=.8\linewidth]{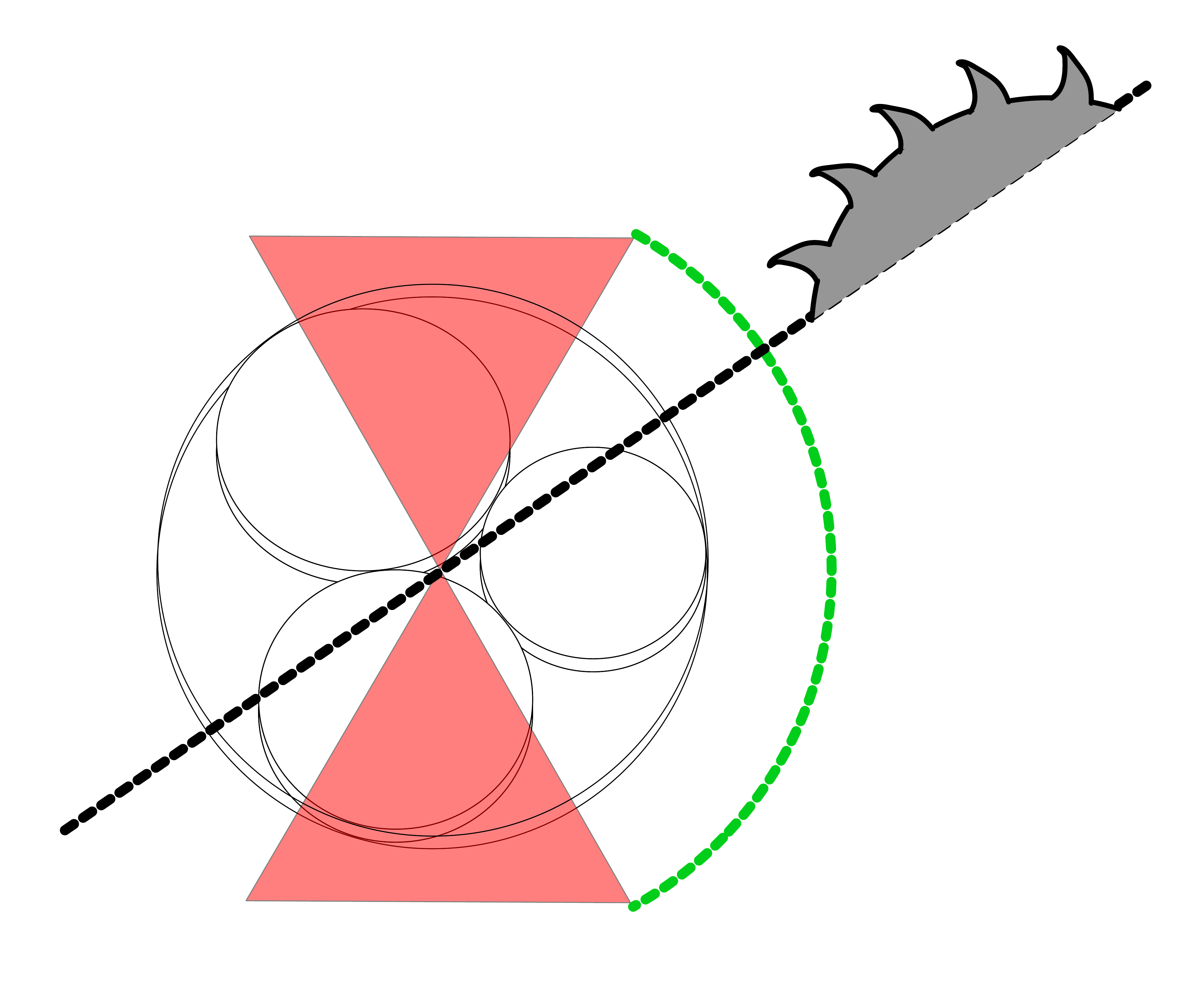}
\end{minipage}
\caption{(Left) A bundle of three logs, $A_i\times I$ for $i=1,2,3$.
Dashed red lines indicate ``horizontal cuts.''
(Right) A ``vertical'' cut through the center of one slice of the log bundle.
The saw blade is on a fixed pivot that can not swivel by an angle of more than $\frac{2\pi}{3}$.}
\label{fig:logs}
\end{figure}

\begin{theorem}[Generalization of Theorem~{\ref{thm:ham-sandwich}}]
\label{thm:logs}
Fix integers $d,n\geq 1$.
Let
\[D^{d+1}=\{(x_1,x_2,\dots, x_{d+1},0)~|~x_1^2+x_2^2+\dots+x_{d+1}^2\leq 1\}\subset \R^{d+2}\]
and suppose $A_1,\dots, A_{n}\subset D^{d+1}$ are measurable sets.
For
\[p\in \partial D^{d+1}=\{(x_1,x_2,\dots, x_{d+1},0)~|~x_1^2+x_2^2+\cdots+x_{d+1}^2 = 1\},\]
let $\pi(p)=\{x\in\R^{d+2}~|~\langle x,p\rangle =0\}$ denote the hyperplane passing through the origin normal to~$p$.
Furthermore, let $H^+(p)$ and $H^-(p)$ denote the (closed) halfspaces of $\R^{d+2}$ determined by the inequalities $\langle x,p \rangle\geq 0$ and $\langle x,p \rangle\leq 0$, respectively.
Then, there exist numbers $0=t_0\leq t_1\leq \cdots \leq t_{k}\leq t_{k+1}=1$ with $k\le n$
and vectors $p_1,\dots,p_{k+1}\in\partial D^{d+1}$ such that 
\begin{enumerate}
\item the vectors $p_i$ are ``close'' in the sense that $\arccos(\langle p_i,p_j\rangle)\leq \pi-p_{n+1}(\R P^d)$ for all $i$ and $j$ (in particular, no two vectors are antipodal), and
\item there exists an equipartition of the $n$ masses $A_j\times I$ given by taking $(A_j\times [t_{i-1},t_i])\cap H^+(p_i)$ for each $1\leq i \leq k+1$.
\end{enumerate}
\end{theorem}

This equipartition means that for each $j$, the sum of the measures of $(A_j\times [t_{i-1},t_i])\cap H^+(p_i)$ is equal to the sum of the measures of $(A_j\times [t_{i-1},t_i])\cap H^-(p_i)$, where both sums vary over $1\leq i \leq k+1$.

\begin{proof}
Define a continuous function $f\colon S^d\to\R^n$ by \[p\mapsto \left(\text{the measure of }A_1\cap H^+(p), \dots, \text{the measure of }A_n\cap H^+(p)\right).\]
By Corollary~\ref{cor:antipodal} there exist vectors $p_1,\dots,p_{k+1}\in S^d$ (with $k\le n$ by Carath{\'e}odory's theorem) and convex coefficients $\lambda_1,\lambda_2,\dots \lambda_{k+1}$ such that $\arccos(\langle p_i,p_j\rangle)\leq \pi-p_{n+1}(\R P^d)$ for all $i$ and $j$ and 
\[\sum_{i=1}^{k+1}\lambda_i f(p_i)=\sum_{i=1}^{k+1}\lambda_i f(-p_i).\]
Observe that 
\[f(-p)=(\text{the measure of }A_1\cap H^-(p),\ \dots,\ \text{the measure of }A_n\cap H^-(p)).\]
Setting $t_i=\sum_{j=1}^i\lambda_j$ for each $0\leq i \leq k+1$ completes the proof.
\end{proof}

\begin{remark}\label{rmk:logs-stronger}
We may obtain a stronger version of Theorem~\ref{thm:logs}, allowing one more set to be equipartitioned, by dropping the requirement that each hyperplane $\pi(p)$ must pass through the origin.
In this setting, $n+1$ sets may be equipartitioned by $n+1$ hyperplanes determined by vectors $p_1,\dots, p_{n+1}$ which again satisfy $\arccos(\langle p_i,p_j\rangle)\leq \pi-p_{n+1}(\R P^d)$ for all $i$ and $j$.
\end{remark}

We will refer to a hyperplane that passes through the origin as a \emph{linear hyperplane}.
A linear hyperplane $H$ is \emph{oriented} if the two halfspaces determined by~$H$ come with a choice of positive halfspace $H^+$ and negative halfspace~$H^-$.
The Ham Sandwich theorem may be restated in the following way:

\begin{theorem}[Ham Sandwich theorem -- restated]\label{thm:ham-sandwich-restated}
Let $\mu_1,\dots,\mu_d$ be finite Borel measures on $S^d \subset \R^{d+1}$ such that every linear hyperplane in $\R^{d+1}$ has measure $0$ for each of the~$\mu_i$.
Then there exists a linear hyperplane~$H$ in $\R^{d+1}$ such that $\mu_i(H^+ \cap S^d)=\frac{1}{2}\mu_i(S^d)$ for $i=1,\dots, d$, where $H^+$ denotes one of the half-spaces defined by~$H$.
\end{theorem}

Theorem~\ref{thm:main} implies the following:

\begin{theorem}
\label{thm:ham-sandwich-spherical}
Let $\mu_1, \dots, \mu_n$ on $S^d\subset \R^{d+1}$ be finite Borel measures such that each linear hyperplane in $\R^{d+1}$ has measure $0$ for each of the~$\mu_i$.
Then there exists a set $\mathcal{H}$ of oriented linear hyperplanes in $\R^{d+1}$ at pairwise angle (in $\R P^d$) at most $\delta=\pi-p_{n+1}(\R P^d)$ such that, for each $i$, at least half of $\mu_i$ is contained in the negative halfspace determined by some hyperplane in $\mathcal{H}$, and at least half of $\mu_i$ is contained in the positive halfspace determined by some hyperplane in $\mathcal{H}$.
\end{theorem}

\begin{proof}
For $x\in S^d\subset \R^{d+1}$, let $H^-(x)$ and $H^+(x)$ denote the (closed) halfspaces of $\R^{d+1}$ determined by the inequalities $\langle y,x \rangle\leq 0$ and $\langle y,x \rangle\geq 0$, respectively.
Let $V$ denote the vector space of odd maps $S^d\to\R$ generated by the vectors $f_i$ defined by $f_i(x)= \mu_i(H^+(x))-\mu_i(H^-(x))$ 
for $1\leq i \leq n$.
Note that $V$ has dimension at most $n$, and that each $f_i$ is odd.
By Theorem~\ref{thm:main}, there exists a set $A\subset S^d$ of diameter at most $\delta$ such that, for every $f\in V$, the restriction $f|_A$ is not strictly positive (and, since $-f\in V$, it is also true that the restriction $f|_A$ is not strictly negative).
Hence, for each $1\leq i \leq n$, there must exist points $a_i,b_i\in A$ such that $f_i(a_i)$ is nonpositive and $f_i(b_i)$ is nonnegative, that is, at least half of $\mu_i$ is contained in the negative halfspace $H^-(a_i)$ and at least half of $\mu_i$ is contained in the positive halfspace $H^+(b_i)$.
So we are done by letting $\mathcal{H}=\{H^\pm(a)~:~a\in A\}$.
\end{proof}

\begin{remark}
\label{rem:ham-sandwich-generalization}
Instead of $\delta=\pi-p_{n+1}(\R P^d)$ we could have phrased both Theorem~\ref{thm:logs} and  Theorem~\ref{thm:ham-sandwich-spherical} for the smallest $\delta$ such that the metric thickening $S^d_\delta$ has cohomological index at least~$n$.
For $n = d$, this $\delta$ is zero, since $S^d$ is a subspace of $S^d_\delta$ for any $\delta \ge 0$.
In the case $n=d$ and $\delta =0$, the stronger version of Theorem~\ref{thm:logs} described in Remark~\ref{rmk:logs-stronger} specializes to Theorem~\ref{thm:ham-sandwich}, and Theorem~\ref{thm:ham-sandwich-spherical} specializes to Theorem~\ref{thm:ham-sandwich-restated}.
\end{remark}

\section{Generalizations of Lyusternik--Shnirel'man--Borsuk covering theorems}\label{sec:covering}

\begin{theorem}[Lyusternik--Shnirel'man--Borsuk covering theorem~{\cite{lyusternik1947topological}}]
\label{thm:Lyusternik}
For any cover $A_1,\dots, A_{d+1}$ of the sphere $S^d$ by $d+1$ closed sets, there is at least one set containing a pair of antipodal points.
\end{theorem}

In fact, the conclusion of Theorem~\ref{thm:Lyusternik} holds for more general collections of subsets of~$S^d$.

\begin{theorem}\label{thm:Lyusternik-gen}
For any cover $A_1,\dots, A_{d+1}$ of the sphere $S^d$ by $d+1$ sets such that the first $d$ sets $A_1,\dots, A_{d}$ are each open or closed, there is at least one set containing a pair of antipodal points.
\end{theorem}

A proof of this theorem appears in~\cite{aigner2010proofs}; see also~\cite{greene2002new,matousek2003using}.
We now consider a generalization in which the number of sets in the covering may be arbitrarily large with respect to the dimension of the sphere.

\begin{theorem}[Generalization of Theorem~{\ref{thm:Lyusternik-gen}}]\label{thm:LSB}
Fix integers $n\ge d\geq 1$ and suppose $A_1,\dots,A_{n+1}$ is a cover of the sphere $S^d$ by $n+1$ sets such that the first $n$ sets $A_1,\dots, A_n$ are each open or closed.
Furthermore, suppose that any subset of the sphere of diameter at most $c_n(S^d)$ is contained in some subset~$A_i$.
In particular, this is satisfied if every subset of the sphere of diameter at most $\pi-p_{n+1}(\R P^d)$ is contained in some~$A_i$.
Then, there is at least one set $A_i$ containing a pair of antipodal points.
\end{theorem}

Recall the $\Z/2$-Carath{\'e}odory number $c_n(S^d)$ is defined as the infimum over all $t\in [0,\infty)$ such that for any odd map $f \colon X \to \R^n$ there is a set $A \subset X$ with $\diam(A) \le t$ and $0 \in \conv(f(A))$.
The above theorem generalizes Theorem~\ref{thm:Lyusternik} because if $n=d$, then the condition that any subset of the sphere of diameter at most $c_d(S^d) = 0$ is in some subset $A_i$ simply implies that the sets $A_i$ cover the sphere.

\begin{proof}
Assume, for the sake of contradiction, that no set in the cover contains antipodal points.
Consider the continuous map $f\colon S^d\to\R_{\ge0}^{n}$ defined by $f(x)=(d(x,A_1),\ldots,d(x,A_{n}))$, where $d(x,A_i)=\inf_{y\in A_i}\{d(x,y)\}$. 
By Corollary~\ref{cor:antipodal}, there exists a subset $\{x_1,\ldots,x_m\}\subset S^d$ (with $m\le n+1$) of diameter at most $c_n(S^d)$ and convex coefficients $\lambda_i> 0$ such that
\begin{equation}\label{eq:vec-y}
y:=\sum_{i=1}^m \lambda_i f(x_i)=\sum_{i=1}^m \lambda_i f(-x_i).
\end{equation}

Since $A_{n+1}$ does not contain antipodal points by assumption, for each $1\leq i \leq m$ at least one of $x_i$ and $-x_i$ must be contained in some element of $\{A_1,\dots, A_n\}$.
In fact, we claim that there must exist a single $A_j\in\{A_1,\dots, A_n\}$ containing all of $\{x_1,\dots, x_m\}$ or all of $\{-x_1,\dots, -x_m\}$.
Toward proving the claim, note that the points $x_1,\dots,x_m$ are all contained in some element of the cover because $\diam(\{x_1,\dots, x_m\})\leq c_n(S^d)$. 
Hence, either $\{x_1,\dots, x_m\}\subseteq A_{j}$ for some $1\leq j \leq n$, or $\{x_1,\dots, x_m\}\subseteq A_{n+1}$.
In the latter case, because $A_{n+1}$ does not contain antipodal points, we have $\{-x_1,\dots, -x_m\}\subseteq S^d\setminus A_{n+1}$.
Then, because $\diam(\{-x_1,\dots, -x_m\}) = \diam(\{x_1,\dots, x_m\})$, it follows that $\{-x_1,\dots, -x_m\}\subseteq A_{j}$ for some $1\leq j \leq n$.
This proves the claim.

Now, observe that either $d(x_i,A_j)=0$ for all $1\le i\le m$ or $d(-x_i,A_j)=0$ for all $1\le i\le m$. 
Furthermore, by considering the $j^\text{th}$ coordinate of $y$ in Equation~\eqref{eq:vec-y} above, it follows that both $d(x_i,A_j)=0$ and $d(-x_i,A_j)=0$ for all $i$.
There are two cases:
\begin{enumerate}
\item Suppose $A_j$ is closed. 
In this case, $d(x_i,A_j)=0$ and $d(-x_i,A_j)=0$ imply that $\{x_i,-x_i\}\subseteq A_j$ for all $i$, contradicting the assumption that no set in the cover contains antipodal points. 
\item Suppose $A_j$ is open. 
Note that $d(-x_i,A_j)=0$ implies $-x_i\in \overline{A_j}$ for all $i$.
In turn, $\overline{A_j}$ is contained in the closed set $S^n\setminus(-A_j)\supseteq A_j$.
Hence, each $-x_i$ belongs to $S^n\setminus(-A_j)$, which implies
that $x_i\notin A_j$ for all $i$.
Swapping the roles of $x_i$ and $-x_i$, a similar argument shows that $-x_i\notin A_j$ for all $i$. 
This contradicts the fact that $A_j$ contains all of $\{x_1,\dots, x_m\}$ or all of $\{-x_1,\dots, -x_m\}$.
\end{enumerate}
Hence some $A_j$ for $1\le j\le n+1$ contains a pair of antipodal points.
\end{proof}

\section{Conclusion}

The presence of a large projective code bounds the topology of the metric thickening of a sphere, which in turn controls the structure of zeros of odd maps from the sphere into Euclidean space.
As a consequence, we obtain a generalization of the Borsuk--Ulam theorem for maps from spheres into higher-dimensional codomains.
We derive consequences of this Borsuk--Ulam theorem for overdetermined versions of the Ham Sandwich theorem (more measures than the ambient dimension), 
for overdetermined versions of the Lyusternik--Shnirel'man--Borsuk covering theorem (more sets than the ambient dimension), 
for generic manifold embeddings,
and for bounds on circular chromatic numbers.

We end with open questions on the topology of metric thickenings of spheres, their concrete relation to the size of projective codes, and on a possible connection to Gromov--Hausdorff distances between spheres.

\begin{question}
Corollary~\ref{cor:high-coindex} asserts that for any  $\delta>\frac{\pi}{2}$, there is an arbitrarily large jump in coindex between $S^d_{\pi/2}\simeq S^d$ and $S^d_\delta$ for $d$ sufficiently large.
Is it possible to determine $\coind(S^d_\delta)$ and $\ind(S^d_\delta)$ as a function of $\delta \in (\frac{\pi}{2}, \pi)$ at least asymptotically in~$d$?
\end{question}

\begin{question}
Theorem~\ref{thm:lower-bound} shows if $p_{n+1}(\R P^d) \ge \pi-\delta$ then $\ind(S^d_\delta) \ge \coind(S^d_\delta) \ge n$.
Does an approximate converse of this result hold?
\end{question}

We only raise the question about an ``approximate'' converse, since $S^d_\delta$ in addition to the size of codes in $\R P^d$ also encodes symmetries of~$S^d$.
For example, an equilateral triangle inscribed in $\R P^1$ gives a projective code with distance~$\frac{\pi}{3}$.
For $\delta = \pi - \frac{\pi}{3}$, Theorem~\ref{thm:lower-bound} then gives an odd map $S^2 \to S^1_\delta$.
However, $S^1_\delta \simeq S^3$, since there are a circle's worth of equilateral triangles in~$\R P^1$.  

\begin{question}
For $d\le n$, could the $\Z/2$-Carath{\'e}odory number $c_n(S^d)$ be equal to $2\cdot d_{\gh}(S^d,S^n)$, twice the Gromov-Hausdorff distance between the spheres $S^d$ and $S^n$, equipped with the geodesic metric?
This is known to be true when $n\le 3$, and is consistent with known bounds on $c_n(S^d)$ (see~\cite[Table page~80]{BushThesis}) and on $d_{\gh}(S^d,S^n)$ (see~\cite[Figure~2]{lim2021gromov}).
Or does one quantity bound the other?
\end{question}

\section*{Acknowledgements}

We would like to thank Chris Cox for a helpful explanation regarding~\cite{bukh2020}, and Sunhyuk Lim for pointing out a correction in Remark~\ref{rem:explicit-code-bounds}.

\bibliographystyle{plain}
\bibliography{TheTopologyOfProjectiveCodesAndTheDistributionOfZerosOfOddMaps.bib}

\end{document}